\documentclass[12pt]{article}
\usepackage{amsmath,amsfonts,amssymb,amsthm}
\newtheorem{thm}{Theorem}[section]
\newtheorem{pro}[thm]{Proposition}
\newtheorem{lem}[thm]{Lemma}
\newtheorem{cor}[thm]{Corollary}
\theoremstyle{definition}
\newtheorem{defn}[thm]{Definition}

\numberwithin{equation}{section}
\newcommand{\CA}{\mathcal{A}}
\newcommand{\CH}{\mathcal{H}}
\newcommand{\CU}{\mathcal{U}}

\newcommand{\CK}{\mathcal{K}}
\newcommand{\CV}{\mathcal{V}}

\newcommand{\Z}{\mathbb{Z}}

\newcommand{\N}{\mathbb{N}}

\newcommand{\cp}{\Delta}

\newcommand{\ket}[1]{|#1\rangle}

\newcommand{\acts}{\triangleright}

\newcommand{\half}{{\tfrac{1}{2}}}
\newcommand{\thalf}{{\tfrac{3}{2}}}
\newcommand{\sthalf}{{\scriptstyle\frac{3}{2}}}
\newcommand{\shalf}{{\scriptstyle\frac{1}{2}}}
\newcommand{\oh}{{\mathchoice{\half}{\half}{\shalf}{\shalf}}}
\newcommand{\thr}{{\mathchoice{\thalf}{\thalf}{\sthalf}{\sthalf}}}

\newcommand{\ts}{\otimes}

\newcommand{\sq}{\unskip\nobreak\kern5pt\nobreak\vrule height4pt width4pt depth0pt}   
\newcommand{\tpi}{\tilde{\pi}}

\title{Twisted Dirac Operators over Quantum Spheres}
\author{Andrzej Sitarz
\thanks{The author acknowledges the Alexander von Humboldt Fellowship
at the Mathematical Institute, Heinrich-Heine-Universit\"at
Universit\"atsstrasse 1, 40225 D\"usseldorf, Germany}
\thanks{Partially supported by MNII Grant 115/E-343/SPB/6.PR UE/DIE 50/2005--2008}\\
Institute of Physics, Jagiellonian University, \\
Reymonta 4, 30-059 Krak\'ow, Poland}
\begin{document}
\maketitle
\begin{abstract}
We construct new families of spectral triples over quantum
spheres, with a particular attention focused on the standard
Podle\'s quantum sphere and twisted Dirac operators.
\end{abstract}

\section{Introduction}
The quantum spaces and, in particular, quantum spheres, are
challenging toy models in noncommutative geometry. The standard
quantum Podle\'s sphere \cite{Podles}, which is a quantum
homogeneous space, was the first object in the realm of
the $q$-deformed manifolds, on which a spectral geometry in
the sense of Connes \cite{Connes} (see also \cite{BFV} for
a review) was constructed \cite{DaSi}. More examples and
local index calculations followed \cite{ADLW,Thrymr,Naiad}.

One of the {\em ad-hoc} assumptions of the \cite{DaSi} construction
was the existence of the $q \to 1$ limit. This, together with the
$\CU_q(su(2))$ equivariance enforced the geometric construction
of the Hilbert space. However, one may wonder whether in the
noncommutative situation we should really be imposing such
restrictions, which refer directly to the classical (commutative)
case. Moreover, we know from examples that in some cases the
axioms of spectral geometry might be satisfied only with certain
accuracy - up to ``infinitesimals'' within the algebra of bounded
operators.

On the other hand, a closer look at the classical situation of the
two-dimensional sphere \cite{MSSV} shows that apart from the
standard Dirac operators, there exists a family of twisted
Dirac operators, with the Hilbert space of spinors twisted
by tensoring it with a line bundle of a nontrivial Chern character.

In this paper we shall explore all these possibilities,
focusing our attention first on the standard Podle\'s sphere.

Our notation throughout the paper is as follows: $0<q<1$ is a
deformation parameter, $[x]$ denotes a $q$-number:
$$ [x] :=\frac{q^x-q^{-x}}{q-q^{-1}},$$
The definitions of the polynomial algebra $\CA(S^2_q)$
and the universal enveloping algebra $\CU_q(su(2))$ use
the standard presentations, for more details we refer
the reader to \cite{DaSi, Naiad}.

\section{The standard Podle\'s sphere \\
and its equivariant spectral geometries.}

We recall here the definition of the algebra $\CA(S^2_q)$
of the standard Podle\'s quantum sphere \cite{Podles} and
its $\CU_q(su(2))$ symmetry.

\begin{defn}
The polynomial algebra of the standard Podle\'s quantum
sphere, $\CA(S^2_q)$, is a star algebra generated by
$B,B^*$ and $A=A^*$, with the relations:
\begin{equation}
\begin{aligned}
AB = q^2 BA, & \phantom{xxx} & A B^* = q^{-2} B^* A, \\
B B^* = q^{-2}A (1 - A), & & B^* B = A (1 - q^2 A).
\end{aligned}
\label{PodSta}
\end{equation}
\end{defn}

The quantized algebra $\CU_q(su(2))$ has $e,f,k,k^{-1}$ as
generators of the $*$-Hopf algebra, satisfying relations:
\begin{equation}
\begin{aligned}
ek = qke,  &\phantom{xxx}& kf = qfk, & \phantom{xxx} & k^2 -
k^{-2} = (q-q^{-1})(fe-ef),
\end{aligned}
\end{equation}
with the standard coproduct, counit, antipode and star:
\begin{equation}
\begin{aligned}
\cp k = k \ts k,  &\phantom{xxx}&\cp e =  e \ts k + k^{-1} \ts e,
&\phantom{xxx}& \cp f =  f \ts k + k^{-1} \ts f \\
\epsilon(k) = 1,  &\phantom{xxx}&  \epsilon(e) = 0,
&\phantom{xxx}& \epsilon(f) = 0, \\
Sk = k^{-1}, &\phantom{xxx}& Sf = - qf, &\phantom{xxx}& Se = -q^{-1} e, \\
 k^* = k, &\phantom{xxx}& e^* =f, &\phantom{xxx}& f^* =e.
\end{aligned}
\end{equation}

{}From the usual Hopf algebra pairing between $\CU_q(su(2))$ and
$\CA(SU_q(2))$, we obtain an action of $\CU_q(su(2))$ on $\CA(SU_q(2))$,
which when restricted to $\CA(S^2_q)$ is given on its generators by:
\begin{equation}
\begin{aligned}
e \acts B&= -q^{-\oh} [2] A + q^{-\thr}, & \;\;\;
e \acts B^*&= 0, \;\;\; &
e \acts A &= q^{-\oh} B^*, \\
k \acts B &= q B, &
k \acts B^* &= q^{-1} B^*,&
k \acts A &= A, \\
f \acts B^* &= q^{-\oh} [2] A - q^{-\oh}, &
f \acts B &=0, &
f \acts A &= - q^{\oh} B.
\end{aligned}
\label{aphereact}
\end{equation}

This action preserves the $*$-structure:
\begin{equation}
h \acts (x^*) = \left( (Sh)^* \acts x \right)^*, \;\ \forall h \in
\CU_q(su(2)), x \in \CA(S^2_q). \label{staract}
\end{equation}

\subsection{Equivariant representations of {$\CA(S^2_q)$}.}

In the next step we extend the results of \cite{DaSi} and find
explicit formulas for a family of equivariant representations
of the algebra $\CA(S^2_q)$ on a Hilbert space $\CH$. The
representations derived here are in fact a restriction of the
family of representation for all Podle\'s spheres \cite{ADLW} to the
case of the standard Podle\'s sphere. For completeness, however,
we recall here the main steps of the construction.

Let us recall, the definition of an equivariant representation:
\begin{defn}
Let $\CV$ be an $A$-module and $H$ be a Hopf algebra. We say that
the representation
$\pi$ of $A$ on $\CV$  is $H$-equivariant if there exists
a representation $\rho$ of $H$ on $\CV$ such that:
\begin{equation}
\label{covar}
 \rho(h) (\pi(a) v) = \pi(h_{(1)} \acts a) \rho(h_{(2)}) v \,
~~~\forall h \in H, ~a \in A, ~v\in \CV.
\end{equation}
where we have used the Sweedler's notation for the coproduct of $H$
and $\acts$ for the action of $H$ on $A$.
\end{defn}

In the construction we use the infinite dimensional linear
space, which after completion shall be Hilbert space of the
spectral triple construction. However, since the representation
of the Hopf algebra of $\CU_q(su(2))$ is unbounded, the
equivariance relation makes sense only on the dense
subspace of $\CH$.

To construct suitable modules $\CV$ we use the (known)
representation theory of $\CU_q(su(2))$. The idea and details
of derivations are similar as in the case of \cite{DaSi} and
\cite{SchWa,ADLW}, therefore we here we present only the result:
\begin{pro}
For each $\pm N \in \Z/2$ (non negative integer or
half-integer) there exists an irreducible
$\CU_q(su(2))$-equivariant representation of the standard
Podle\'s quantum sphere on the space $\CV_N$:
\begin{equation}
\CV_N = \bigoplus_{j= |N|,|N|+1,\ldots} V_j,
\end{equation}

where $V_j$ is $2j+1$-dimensional space with the fundamental
representation of $\CU_q(su(2))$ of rank $j$.

The representation $\pi_N$ of $\CA(S^2_q)$ is given on
the basis vectors $\ket{l,m} \in V_l$, $l=|N|,|N|+1,\ldots$,
$m = -l, -l+1, \ldots, l-1,l$, by:
\begin{equation}
\begin{aligned}
\pi_N(B) \ket{l,m} =&
\phantom{+}q^{m} \sqrt{ [l+m\!+\! 1][l+m+2]} \; r^+(l)\; \ket{l\!+\! 1,m\!+\! 1} \\
&+ q^{m} \sqrt{ [l+m\!+\! 1][l-m]} \;r^0(l) \;\ket{l,m\!+\! 1} \\
&+ q^{m} \sqrt{ [l-m][l-m\!-\! 1]} \;r^-(l) \;\ket{l\!-\! 1,m\!+\! 1}, \\
\pi_N(B^*) \ket{l,m} =&
\phantom{+}q^{m\!-\! 1} \sqrt{ [l-m+2][l-m\!+\! 1]} \; r^-(l\!+\! 1)\; \ket{l\!+\! 1,m\!-\! 1}\\
&+q^{m\!-\! 1} \sqrt{ [l+m][l-m\!+\! 1]} \; r^0(l) \;\ket{l,m\!-\! 1} \\
&+q^{m\!-\! 1} \sqrt{ [l+m][l+m\!-\! 1]} \; r^+(l\!-\! 1)\;\ket{l\!-\! 1,m\!-\! 1}, \\
\pi_N(A) \ket{l,m}  =& -q^{m+l+\frac{1}{2}} \sqrt{ [l-m\!+\! 1][l+m\!+\! 1]} \; r^+(l) \; \ket{l\!+\! 1,m} \\
&+ \frac{q^{-\oh}}{1+q^2} \left( ([l\!-\!m\! +\!1][l\!+\!m] \right. \\
& \phantom{1+q^2} \;\;\;\;\;
-\left. q^2 [l\!-\!m][l\!+\!m\!+\!1]) \; r^0(l) \!+\! q^\oh\!
\right)\; \ket{l,m} \\
&+ q^{m-l-\oh} \sqrt{ [l-m][l+m]} \; r^-(l) \; \ket{l\!-\! 1,m},
\end{aligned}
\label{eqreps}
\end{equation}
where $r^+(l),r^-(l), r^0(l)$ are:
\begin{equation}
\begin{aligned}
r^0(l) &= q^{-\oh}
\frac{(q-\frac{1}{q})[l+|N|+1][l-|N|] \pm q^{\pm 1}[2|N|]}{[2l][2l+2]}, \\
r^+(l) &= \frac{q^{-l-\frac{3}{2}-N}}{[2l+2]}
\frac{\sqrt{[l+N+1][l-N+1]}}{\sqrt{[2l+1][2l+3]}}, \\
r^-(l) &= - q^l r^+(l-1).
\end{aligned}
\label{eqreps2}
\end{equation}
\end{pro}
\begin{proof}
The proof is entirely technical, the derivation of the formulae
could be divided into two steps. First, the use of equivariance
gives the dependence on the $m$ parameter (\ref{eqreps}).
Then, using the defining relations of the algebra
(\ref{PodSta}) leads to two recurrence relations for
$r^+(l),r^0(l)$:
$$
\begin{aligned}
r_+(l)^2 [2l\!+\!1] &[2l\!+\!3] q^{2l+2} (1+q^2)^2
+ r_0(l)^2 [2l]^2 q^2 \\
& + r_0(l) [2l] q^{\frac{1}{2}} (q^2\!-\!1) - q = 0,
\end{aligned}
$$
and
$$
\begin{aligned}
r_+(l)^2 &\frac{(1+q^2)^2}{1-q^2} [2l+3] q^{4l+4}-
  r_+(l-1) \frac{(1+q^2)^2}{1-q^2} q^{2l+2} [2l-1] \\
& + r_+(l)^2 [2l+1] [2l+3] q^{4l+4} (1+q^2)^2
 + r_0(l)^2 [2l]^2 q^{2l+4} \\
&+ r_0(l) [2l] q^{\frac{1}{2}} (q^2-1) q^{2l+2} - q^{2l+3} = 0,
\end{aligned}
$$
By solving them and imposing the boundary conditions (that is
$l \geq |N|$ we obtain the solutions (\ref{eqreps2})).
\end{proof}

\subsection{Twisted Dirac operators.}

Let $\CH_{N,1}$ be the completion of the space
$\CV_N \oplus \CV_{N+1}$, for any $N \in \Z/2$. We
take the diagonal representation $\pi_N \oplus \pi_{N+1}$
and the natural grading operator taken as $1$ on the first
and $-1$ on the second component. We have:

\begin{pro}
The following densely defined operator
\begin{equation}
D_N \ket{l,m,\pm} =  \sqrt{[l-N]}\sqrt{[l+N+1]} \ket{l,m,\mp},
\label{twistDir}
\end{equation}
anticommutes with $\gamma$, has bounded commutators with the
elements of the algebra $\CA(S^2_q)$ and satisfies exactly
the order-one condition. The eigenvalues of $D$ are:
$$ \lambda_D = \pm \sqrt{[l-N] [l+N+1]},\; N>0, l=N,N+1,\ldots$$
and
$$ \lambda_D = \pm \sqrt{[l-N] [l+N+1]},\; N<0, l=|N|-1,|N|,\ldots$$
\end{pro}
with multiplicities $2l+1$. Note that the kernel of $D$ has
dimension $2N+1$ if $N \geq 0$ and $2|N|-1$ for $N<0$. Thus,
the standard Dirac operator $N=-\oh$ has the
eigenvalues $[l+\oh]$ and an empty kernel.

The anticommutation with $\gamma$ is an obvious consequence of the
definition. We shall postpone the proof of bounded commutators with
the elements of the algebra $\CA(S^2_q)$ until the next section,
here we shall concentrate on the order-one condition.

\begin{defn}
We say that the spectral triple $(\CA,\CH,D)$ extends to a real spectral
triple if there exists a spectral triple $(\CA,\CH',D')$ and
an antiunitary isometry $J: \CH \oplus \CH' \to \CH \oplus \CH'$
such that: $(\CA, \CH \oplus \CH', D \oplus D', J)$ is a real
spectral triple.
\end{defn}
As a consequence we can postulate:
\begin{defn}
The Dirac operator $D$ satisfies the order one condition if
there exists a real extension of the spectral triple with
a Dirac operator satisfying the order one condition.
\end{defn}

Coming back to our situation of the family of spectral triples
over the standard Podle\'s quantum sphere, we have:

\begin{lem}
Let $(\CA(S^2_q), \CH_{N,1}, D_N)$ be the spectral triple as defined
above. Then, $(\CA(S^2_q), \CH_{-N-1,1}, D_{-N-1})$ extends it
to a real spectral triple and the operator $D_N$ satisfies the
order-one condition.
\end{lem}

\begin{proof}
We define first the reality operator $J$:
\begin{equation}
J \ket{l,m,\pm}_N = i^{2m} \ket{l,-m,\mp}_{-N-1},
\;\;\;\; \ket{l,m,\mp}_K \in \CH_{K,1}.
\label{defJ}
\end{equation}
$J$ is well-defined, since the subspace of $\CH_N$ with eigenvalue
of $\gamma$ $+1$ is $\CV_N$, whereas the subspace of $\CH_{-N-1}$
with $\gamma$ eigenvalue $-1$ is $\CH_{-N}$. It is easy to check
that
$$ J \gamma = - \gamma J, $$
and
$$ J^2 = \pm 1. $$
Note that only for a half-integer $N$ we have the signs of
a two-dimensional spectral geometry, whereas for an integer
value of $N$, we have the sign relations corresponding
formally to a six-dimensional (modulo $8$) real structure.

Further, we check that for any two generators $x,y$
of the algebra $\CA(S^2_q)$:
$$ \left[ J^{-1} \pi(x) J, \pi(y) \right] =0.$$
This could be verified either with an explicit calculations
or using the arguments, which identify the space $\CH_N$
with a certain subspace of the Hilbert space of the GNS
construction for $SU_q(2)$ and $J$ with a conjugation map.

Finally, explicit calculations show that the Dirac operator
$D_N$ satisfies the order-one condition in the sense of
extension to a real spectral triple. We skip the lengthy
presentation of the calculations\footnote{The symbolic
calculations are available from the author}.
\end{proof}

\subsection{More families of spectral triples over $\CA(S^2_q)$}

In the previous section we have found a good candidate for
the equivariant twisted Dirac operator over the standard
Podle\'s quantum sphere. Still, we need to prove that it
has bounded commutators with the elements of the algebra.

In this part, in addition to the above construction, we
shall briefly sketch the construction of more families of
equivariant spectral triples over $\CA(S^2_q)$, which satisfy
the geometric conditions up to the ideal of compact operators.
The boundedness of commutators with the elements of the algebra
shall be much easier to show using the approximate representations,
so the special case of twisted Dirac operators shall follow
as a corollary.

Similarly as in the construction for $SU_q(2)$ we first
define the ideal $\CK_q$ as an ideal of operators of
exponential decay. $\CK_q$ could be viewed as an ideal
generated by a diagonal operator on $\CH_N$ with spectrum $q^l$.

Our main tool (as in \cite{Naiad}) is the approximate
representation of the algebra $\CA(S^2_q)$:

\begin{pro}
The maps $\tilde{\pi}_N$:
\begin{equation*}
\begin{aligned}
\tpi_N(B) \ket{l,m}_N &=
\phantom{+} \sqrt{1\!-\!q^{2(l+m+2)}}\sqrt{1\!-\!q^{2(l+m\!+\! 1)}} q^{l-N} \;\ket{l\!+\! 1,m\!+\! 1}_N &&\\
&+ q^{l+m} \sqrt{1\!-\!q^{2(l+m\!+\! 1)}}  \;\ket{l,m\!+\!1}_N &&\\
&- q^{2(l+m)} q^{l-N} \;\ket{l\!-\! 1,m\!+\! 1}_N + o(q^{2l}), &&\\
\tpi_N(B^*) \ket{l,m}_N =
&- q^{2(l+m)\!+\! 1} q^{l-N} \;\ket{l\!+\! 1,m\!-\!1}_N&&\\
& + q^{l+m\!-\! 1} \sqrt{1\!-\!q^{2(l+m)}} \;\ket{l,m\!-\!1}_N &&\\
& + q^{l-N\!-\! 1} \sqrt{1\!-\!q^{2(l+m)}} \sqrt{1\!-\!q^{2(l-m\!-\! 1)}}
 \;\ket{l\!-\!1,m\!-\!1}_N \!+\! o(q^{2l}), &&\\
\tpi_N(A) \ket{l,m}_N
&= -q^{l+m} q^{l-N\!+\! 1} \sqrt{1\!-\!q^{2(l+m\!+\! 1)}} \;\ket{l\!+\!1,m}_N &&\\
 &+ q^{2(l+m)} \;\ket{l,m}_N &&\\
 &+ q^{l+m} q^{l-N\!-\! 1} \sqrt{1\!-\!q^{2(l+m\!-\! 1)}} \;\ket{l\!-\!1,m}_N + o(q^{2l}),&&
\end{aligned}
\end{equation*}
give an approximate representations of $\CA(S^2_q)$ on the modules
$\CV_N$, that is for any $x\in\CA(S^2_q)$ the difference
$\pi_N(x) - \tilde{\pi}_N(x)$ is in $\CK_q$.
\end{pro}

It is worth noting that, actually we do have a bit more than an approximate
representation, as the above formulas give the approximate representation
up to order $q^{2l}$, so that the difference $\pi(x)-\tilde{\pi}(x)$
is of order at least $q^{2l}$. This shall be important in the
calculations concerning the commutators with an unbounded Dirac
operator. We have:

\begin{pro}
Let $\CH_{N,r}$ for $N \in \Z/2$ and $r \in \N$ be the completion of
$\CV_N \oplus \CV_{N+r}$. Let $\gamma$ be the natural $\Z_2$ grading
taken as $1$ on the first component and $-1$ on the second. We denote
by $\pi$ the diagonal representation of $\CA(S^2_q)$ on $\CH_{N,r}$
and by $\tilde{\pi}$ its approximate representation.

The operator:
\begin{equation}
D_K \ket{l,m}_K =
\begin{cases} 0 & \hbox{if\ } l < |K+r| \\
d_K q^{-l} \ket{l,m}_{K'}, & \hbox{if\ } l \geq |K+r|.
\end{cases}
\label{twistDirg}
\end{equation}
where $K=N,K'=N+r$ or $K=N+r,K'=N$, and $d_K$ are complex coefficients
has bounded commutators with the algebra $\CA(S^2_q)$, anticommutes
with $\gamma$ and satisfies the order one condition up to the ideal
$\CK_q$ (in the sense of the real extension of a spectral triple).
$D$ is selfadjoint if and only if $(d_N)^*=d_{N+r}$.
\end{pro}

\begin{proof}
First, let us check that the commutators with the elements of the
algebra are bounded. We calculate, for example:
$$
\begin{aligned}
& \left(  D_{N} \tilde{\pi}_{N}(B) - \tilde{\pi}_{N+r}(B) D_{N} \right)
\ket{l,m,+}_{N,r} && \\
&= \left( d_{N} q^{-l\!-\! 1} q^{l-N} - d_N q^{-l} q^{l-N-r} \right)
   \sqrt{1\!-\!q^{2(l+m+2)}}\sqrt{1\!-\!q^{2(l+m\!+\! 1)}} \ket{l\!+\! 1,m\!+\! 1}_{N+r} &&\\
&+ \left( d_{N} q^{-l} - d_{N} q^{-l} \right)
q^{l+m} \sqrt{1\!-\!q^{2(l+m\!+\! 1)}}  \ket{l,m\!+\!1}_{N+r} &&\\
&- \left( d_{N} q^{-l\!+\! 1} q^{l-N} - d_{N} q^{-l} q^{l-N-r} \right)
q^{2(l+m)} \ket{l\!-\! 1,m\!+\! 1}_{N+r} + o(q^{2l}), &&
\end{aligned}
$$
and it is easy to see that the expression remains bounded.

Next, for the real extension of the spectral triple, we notice
that the real extension of $\CH_{N,r}$ is $\CH_{-N-r,r}$, with
$J$ being the antilinear isometry between $\CV_N$ and $\CV_{-N}$,
and between $\CV_{N+r}$ and $\CV_{-N-r}$. Thus, the same arguments
as in the case $r=1$ from the previous section apply.

Finally, for the order one condition, we can use the following
argument. Each of the generators $A,B,B^*$ is of the type
$T_0 + T_q$, where $T_0$ is bounded, $[T_0,D]=0$
and $T_q \in \CK_q$ is an operator such that both
$D T_q$ and $T_q D$ are bounded. Then:
$$
\begin{aligned}
& \left[ J (T_0^x + T^x_q) J^{-1}, [D, (T_0^y + T^y_q)] \right] = \\
&= \left[ J (T_0^x + T^x_q) J^{-1}, [D, T^y_q] \right] \in \CK_q.
\end{aligned}
$$
where in the last estimation we have first used
that $[D, T^y_q]$ must be at most bounded and
$J T^x_q J^{-1} \in \CK_q$. To estimate the commutator
of $J T_0^x J$ with $[D, T_q^y]$, we observe first
that since $D$ commutes with $T_0^x$, it shall be
sufficient to prove that $[[J T_0^x J^{-1}, T_q^y],D]$
is in $\CK_q$.

First, take, for instance the elements, for which the
compactness of the commutator is least evident:
$$
\begin{aligned}
T_0^A \ket{l,m} &= q^{2(l+m)} \ket{l,m}, \\
T_q^{B+} \ket{l,m} &=  q^{l-N} \sqrt{1-q^{2(l+m+2}}
\sqrt{1-q^{2(l+m+1}} \ket{l+1,m+1}.
\end{aligned}
$$
We calculate:
$$
\begin{aligned}
[J T_0^A J^{-1}, & T_q^{B+}] \ket{l,m} = \\
=&q^{2(l-m)} q^{l+m} \sqrt{1-q^{2(l+m+2}} \sqrt{1-q^{2(l+m+1}} \ket{l+1,m+1} \\
&- q^{l+m} \sqrt{1-q^{2(l+m+2}} \sqrt{1-q^{2(l+m+1}} q^{2(l-m)} \ket{l+1,m+1} \\
=& 0 + o(q^{2l}).
\end{aligned}
$$
Hence, also the order one condition is satisfied but only up
to compact operators.

For the other commutators it is worth noting that $J T_0 J^{-1}$ has
always a factor $q^{l-m}$. When multiplied by $q^{l+m}$, (which is
present in all $T_q^y$ apart from the above case of $T_q^y=T_q^{B+}$),
we obtain that their product (and hence the commutator) is of order
$q^{2l}$ as most. Therefore multiplying it by $D$ (or taking
a commutator with $D$) still gives a result in $\CK_q$.
\end{proof}

We have shown in this section that the Dirac operator with the
eigenvalues growth $q^{-l}$ satisfies the modified conditions
of (real) spectral geometry. Clearly, a compact perturbation of
such $D$ shall satisfy it as well.

Therefore, for $r=1$, which is the case of previously studied
twisted Dirac operators we have:

\begin{cor}
The twisted Dirac operator (\ref{twistDir}) is a compact
perturbation of the Dirac operator (\ref{twistDirg}) and therefore
has bounded commutators with the elements of the algebra $\CA(S^2_q)$.
\end{cor}

This follows directly from the estimate, true for big $l$:

$$ \sqrt{[l+N]} \sqrt{[l-N+1]} \sim q^{-l}
\left(1 + q^l + \cdots \right).$$

Note that in the $q \to 1$ limit only the twisted case $r=1$
gives a family of Dirac operators and spectral triples over
the two-dimensional sphere, whereas the exotic spectral
triples $r>1$ give quasi-Dirac operators studied in
\cite{Sitarz-qd}.

\subsection{The Fredholm modules and index pairing}

Although the entire construction of the spectral geometry
is very similar to the one already presented in \cite{DaSi},
we shall see that the obtained spectral geometries fall into
a different $K$-homology class. We shall look at the Fredholm
module arising from $(\CA(S^2_q), \pi, \CH_{N,r}, D, \gamma)$
and prove by explicit calculations that the index pairing with
an element form $K_0$ group of the standard Podle\'s sphere
depends on $N$ and $r$.

We shall divide the proof into two parts, first, when
$\ker D \!=\! \emptyset$ and then for the situation when there
are harmonic "spinors". To begin with we consider the
Fredholm modules for "quasi-Dirac" operators, which corresponds
to the case $r=-2N$, where the kernel of $D$ is empty. Using
the sign of the Dirac operator $F = D |D|^{-1}$ we have:

\begin{lem}
The commutators $[F, \pi(a)]$ are trace class for every
$a \!\in\! \CA(S^2_q)$. The pairing between the cyclic
cocycle associated via Chern map to the Fredholm module
$(\CA(S^2_q), \CH_N \oplus \CH_{-N}, F, \gamma)$ and the
nontrivial projector $e$ of the $K_0(\CA(S^2_q))$:
$$
e = \frac{1}{2} \left(
\begin{array}{cc}
(1-A) & qB \\ qB^* & q^2 A
\end{array}
\right),
$$
is $-2N$.
\end{lem}
\begin{proof}
By looking at the approximate representation $\tpi$ we have
already noticed that the difference of representations for
different values of $N$ is always in $\CK_q$, hence it is
trace class.

Therefore, the following expression gives a 0-cyclic cocycle:
\begin{equation}
\phi(x) = \hbox{Tr} \; \gamma F [F,\pi(x)].
\end{equation}
We calculate the paring $<\phi,e>$ explicitly, calculating
the trace over $\CH_N$:
$$
\begin{aligned}
<\phi,e> =& - \hbox{Tr} (1-q^2) (\pi_N(A) - \pi_{-N}(A)) \\
=& - \frac{1\!-\!q^2}{1\!+\!q^2} (q\!+\!\frac{1}{q}) \frac{[2N]}{q}
\sum_{l=N}^\infty \sum_{m=-l}^l
\left(\frac{[l\!-\!m\!+\!1][l\!+\!m]\!-\! q^2[l\!-\!m][l\!+\!m\!+\!1]}{[2l][2l+2]} \right) \\
=& - \frac{2}{q} [2N]  (1-q^2) \sum_{l=N}^\infty l
\left( \frac{q^{2l}}{1-q^{4l}} - \frac{q^{2l+2}}{1-q^{4l+4}} \right) \\
& +  2 q [2N] (1-q^2) \sum_{l=N}^\infty \frac{q^{2l}}{1-q^{4l+4}} = \ldots
\end{aligned}
$$
First, let us call $\frac{q^{2l}}{1-q^{4l}} = a_l$. Then the first
component of the sum is:
$$
\begin{aligned}
- \frac{2}{q}  [2N] & (1-q^2) \sum_{l=N}^\infty l (a_l - a_{l+1}) \\
& = - \frac{2}{q}  [2N] (1-q^2) \left( N a_N + \sum_{l=N+1}^\infty
a_l \right).
\end{aligned}
$$
whereas the second component is identified as:
$$ \frac{2}{q}  [2N] (1-q^2)
\sum_{l=N+1}^\infty \frac{q^{2l}}{1-q^{4l}} =
\frac{2}{q} [2N] (1-q^2) \sum_{l=N+1}^\infty a_l.$$
Therefore we obtain:
$$
\begin{aligned}
<\phi,e> =& -\frac{2}{q} [2N] (1-q^2) N \frac{q^{2N}}{1-q^{4N}} \\
= & -2N.
\end{aligned}
$$
Hence the pairing depends on $N$, which means that the Fredholm
modules obtained for different choices of $N$ are from different
K-homology classes.
\end{proof}

In the less trivial case of the twisted Dirac operators,
we first need to define a proper Fredholm module. Using the
Dirac operator $D$ and defining $F$ as $0$ on the kernel of $D$
and the sign of $D$ the orthogonal complement of the kernel,
we obtain only a pre-Fredholm module, with the relation $F^2=1$
satisfied up to finite rank operator.

Using the procedure of Higson \cite{Higson}, we introduce the Fredholm
module with a doubled Hilbert space $\CH_{N,r} \oplus \CH_{N,r}$
and the following representation, grading as well as
the Fredholm operator $F'$:

$$ \pi'(x) = \left( \begin{array}{cc} \pi(x) & 0 \\ 0 & 0 \end{array} \right),
\;\;\;
\gamma'= \left( \begin{array}{cc} \gamma & 0 \\ 0 & -\gamma \end{array} \right),
\;\;\;
F'= \left( \begin{array}{cc} F & K \\ K & -F \end{array} \right),
$$
where $K$ is the orthogonal projection on the kernel of $D$ and $F$
is the sign of $D$ (taken zero on the kernel of $D$).

It is easy to verify that this is indeed a Fredholm module, and
it again satisfies that $[F', \pi'(x)]$ is trace class for
any $x \in \CA(S^2_q)$. The formula for the cyclic cocycle associated
with this Fredholm module reads now:
\begin{equation}
\phi_{N,r}(x) = \hbox{Tr} \; \gamma \left( F [F,\pi(x)]
+ K \{ K, \pi(x) \} \right),
\end{equation}
where the trace is now reduced back to the original
Hilbert space $\CH_{N,r}$.
\begin{lem}
The pairing between the $K$-homology class defined by the
generalized twisted Dirac operator and the projector $e$
depends on $n$ and $r$:
$$
\langle \phi_{N,r}, e \rangle =
\begin{cases}
-2(N+r - (r-1)(2N+r)),  & N>0 \\
-2(N+r - (r-1)(2N+r)),  & N<0, N+r \geq |N|, \\
 2 r  (r+2N+1), & N < 0, 0 < N+r < |N|, \\
 2(r+1)(2N+r), & N < 0, N+r \leq 0,
\end{cases}
$$
where we always assume $r>0$.
\end{lem}
\begin{proof}
The direct and explicit proof, which we have presented in
the special case of invertible $D$ is too complicated from
the technical point of view. We shall use, however, the result
that the pairing should be independent of $q$. Thus, exploring
the pairing in the $q=0$ limit and assuming that the series,
which define the value of the paring converge uniformly in $q$
and the limit exists, we shall be able to make the
explicit calculations.

First, we need to recover the $q=0$ limit of the diagonal
matrix element of the appropriate representation of $A$.
We have:

$$
\lim_{q \to 0} \langle l,m | \pi_N(A) | l, m \rangle
= \begin{cases} 0 & l > N, l> m, \\
1 & l > N, l=m, \\
1 & l = N. \end{cases}
$$

To calculate the paring $<\phi,e>$ explicitly, we need
to consider the relative signs of $N+r$ and $N$. Since
we can always assume $r>0$, we might have $N+r$ and $N$
of the same sign and of different sign. Our choice of
$\gamma$ is that it is $+1$ on $\CH_{N+r}$ and $-1$
on $\CH_N$. Take, for example $N>0$, then:
$$
\begin{aligned}
<\phi_{N,r},e> =& - \hbox{Tr}
\left( ( \pi_{N+r}(A) - \pi_{N}(A)) (1-K) +
2 K \pi_N (1-A) \right) \\
=& - \left(  2 (N+r) + - (-4 N+4 N r+2 r^2-2 r) \right) \\
=& - \left(  2(N+r - (r-1)(2N+r)) \right).
\end{aligned}
$$
Here the trace is taken over $\CH_N$, then we use the
identification $(1-K) \CH_N \sim \CH_{N+r}$. In the most
interesting case $r=1$, which corresponds to the twisted
Dirac operators, we have the value of the
paring $-2N-2$.

Similarly, one can consider remaining cases, for instance,
if $N<0$ but $N+r >|N|0$, we again have:
$$
\begin{aligned}
<\phi_{N,r},e> =& - \hbox{Tr} \left(
  (1-K) (\pi_{N+r}(A) - \pi_{N}(A)) - 2 K \pi_N(1-A) \right) \\
= & - 2 \left( (N+r - (r-1)(2N+r)) \right).
\end{aligned}
$$
\end{proof}

\section{The twisted Dirac operators \\
         over other Podle\'s quantum spheres}

For the other Podle\'s spheres one cannot expect (as it was shown
first in  \cite{DLPS} then in \cite{ADLW}) the exactness of commutator
relations of the commutant and the order one condition. Therefore,
from the beginning we can work with the approximate representation,
that is with the representation up to compact operators from the
ideal $\CK_q$. The exact families of equivariant representations
were already derived and presented in \cite{ADLW}, here we present
only the approximate ones. The notation is as in the previous part
of the paper, however, to distinguish the case from the standard
Podle\'s quantum sphere we denote its generators by $b,b^*,a$.
We have:
\begin{equation*}
\begin{aligned}
\tpi_N(a) \ket{l,m} =& \phantom{+}
   s q^{l+m}\sqrt{1-q^{2(l+m\!+\! 1)}} \; \ket{l\!+\! 1,m} \\
& +s q^{l+m\!-\! 1} \sqrt{1-q^{2(l+m)}} \; \ket{l\!-\! 1,m} \\
&+ q^{2(l+m)} \; \ket{l,m}, \\
\tpi_N(b) \ket{l,m} =& \phantom{+}
   s \sqrt{1-q^{2(l+m+2)}} \sqrt{1-q^{2(l+m\!+\! 1)}} \; \ket{l\!+\! 1,m\!+\! 1} \\
&- s \delta_{lL} q^{2(l+m)\!+\! 1}  \; \ket{l\!-\! 1,m\!+\! 1} \\
&+ q^{l+m\!+\! 1} \sqrt{1-q^{2(l+m\!+\! 1)}} \; \ket{l,m\!+\! 1}, \\
\tpi_N(b^*) \ket{l,m} =& \phantom{+}
s \sqrt{1-q^{2(l+m)}} \sqrt{1-q^{2(l+m\!-\! 1)}} \; \ket{l\!-\! 1,m\!-\! 1} \\
&+ q^{l+m} \sqrt{1-q^{2(l+m)}} \; \ket{l,m\!-\! 1}  \\
&-s q^{2(l+m)\!+\! 1} \; \ket{l\!+\! 1,m\!-\! 1},
\end{aligned}
\end{equation*}
which satisfy the relations:
\begin{equation*}
\begin{aligned}
ab = q^2 ba, & \phantom{xxxxxxxxxxxxxxxx} & a b^* = q^{-1} b^* a, \\
bb^* +a^2 -a  = s^2, & \phantom{xxxxxxxxxxxxxxxx} &
b^*b+ q^4 a^2 -q^2 a = s^2,
\end{aligned}
\end{equation*}
which is a slightly rescaled version of the family
of Podle\'s spheres, here the $s=0$ limit gives the Standard
Podle\'s sphere whereas $s=\infty$ limit (with an appropriate
rescaling of the generators) gives the equatorial one.

Note that the representation is a'priori defined on the linear
space $\CV'$ of all $\ket{l,m}$, $l \in \Z/2, -l \leq m \leq l$
(the $\delta_{lL}$ term takes care that $l \geq 0$ when $L=0$
or $l \geq \oh$ when $L=\oh$).

Next, we take the definition of the Hilbert space $\CH_N$, $N>0$
as in the previous section and define a representation $\tilde{\pi}_N$
in the following way on the generators $x=a,a^*,b$:
$$ \tilde{\pi}_N(x) \ket{l,m}
= P_N \tilde{\pi}(x) \ket{l,m},$$
where $P_N$ is a projection from the space $\CV'$ onto $\CH_N$.
Since this projection has a kernel of finite rank, the maps
$\tilde{\pi}(x)$ and $\tilde{\pi}_N(x)$ differ for any $x$ by
a finite rank operator, therefore $\tilde{\pi}_N$ is again
an approximate representations.

We have then:
\begin{pro}
Let $\CH_{N,1}$ be the Hilbert space defined as in the previous
section: $\CH_{N,1} = \CH_N \oplus \CH_{N+1}$, with the grading
$\gamma$ as defined before.

Then the following densely defined operator:
\begin{equation}
D_N \ket{l,m}_{N \pm K} = (l-N) \ket{l,m}_{N+1\mp K}, \;\;\;
K=0,1.
\label{twistDirall}
\end{equation}
has bounded commutators with the algebra elements and satisfies
the generalized order one condition up to compact operators from
the ideal $\CK_q$.

The eigenvalues of the Dirac operator are $\pm (l-N)$,
$l=N,N+1,\ldots$ with multiplicities $2l+1$. Note that, in
comparison to the "classical" twisted Dirac operator, $D_N$
is just a compact perturbation, since:
$$ \sqrt{l-N} \sqrt{l+N+1} \sim (l-N)
\left( 1 + \frac{N+\oh}{l-N} + \cdots \right)$$
for $l-N$ sufficiently big.
\end{pro}

The proof of the order-one condition is purely based on the explicit
calculations, which we omit. For the boundedness of commutators, first
of all, we observe that since the approximate representation does not
depend on $N$ (for $l>N$) then the commutators of $D$ with the
operator $T_\pm$, where $T_\pm$ changes
$l$ by $\pm 1$ are $\pm T_\pm$, whereas the commutators with $T_0$ (operators
leaving invariant subspace with fixed $l$) vanish. Since all generators are
just finite sums of such bounded operators, hence their commutators with $D$
are bounded as well.

Interestingly enough, the choice $l-N$ is not the only possibility in our case.
In fact any linear function of $l$ and $m$ satisfies both the
boundedness of commutators as well as the order one condition (up to
the compact operators):
\begin{equation}
\tilde{D}_N \ket{l,m}_{N \pm K} = \delta_{lN}
(l+\alpha_K m) \ket{l,m}_{N+1\mp K}, \;\;\; K=0,1.
\label{twistDirall2}
\end{equation}
where $\alpha_K = \alpha_{1-K}^*$, $K=0,1$ so that $\tilde{D}_N$ is
selfadjoint. Note that such operator is clearly not equivariant with
respect to the $\CU_q(su(2))$ symmetry.

It could be easily checked that this operator (and its natural extension
in the sense of extension of the spectral triple to a real one) satisfies
the order one condition up to the compact operators. However, since
the eigenvalues are $\pm | l+ \alpha_0 m | $ (for $l>N$) only for some
values of $\alpha_0$ the operator $\tilde{D}$ has a compact resolvent.
A good example of the possible choice of $\alpha_0$ is a pure imaginary
number $\alpha_0=i$. Then the eigenvalues are $\sqrt{l^2+m^2}$ and we
can possibly think of such geometry as corresponding rather to a
quantum ellipsoid (roughly speaking) than to a quantum sphere.
\section{Conclusions}

We have shown in this paper that on the standard Podle\'s quantum
sphere there exist a family of equivariant spectral triples,
which are topologically inequivalent and we have extended the
construction for other Podle\'s spheres. Again, similarly as in
the "standard" situation, we see that the spectral properties
of the Dirac operator are very different for the standard
Podle\'s compared with the rest of the family.

Certainly only some of the presented constructions lead in
the $q \to 1$ limit to the classical Dirac operator: this is
the case of twisted Dirac operators. The other constructions have
only commutative shadows in form of {\em quasi-Dirac} geometries
\cite{Sitarz-qd}. For the standard Podle\'s sphere the twisted
Dirac operators are again singled out by the fact that the
order-one condition is satisfied exactly.

{}From the point of view of abstract spectral geometry the existence
of the classical limit shall not be an argument to disqualify certain
geometry. Therefore we might be forced to accept that in the $q$-deformed
case we might have in some situations {\em many} possible topological
sectors admitting a geometry in this sense. This is certainly true in
the situation of the other members of the family of quantum spheres, where in
any case we can satisfy the axioms only up to the ideal compact operators.

A challenging project is to find the description of the twisted Dirac
operators using the description, which appears natural in the classical
commutative geometry: that is with the help of the "standard" Dirac
operator and connection on the line bundle, which twists the spinor
bundle.

\end{document}